\documentclass[12pt]{amsart}
\usepackage[a4paper,twoside,left=2cm,right=2cm,top=3.1cm,bottom=2.3cm]{geometry}

\usepackage{amsmath,amssymb,amscd,amsthm}
\usepackage{latexsym}
\usepackage{graphicx}
\usepackage{babel}
\usepackage[latin1]{inputenc}       
\usepackage{textcomp}
\usepackage{times}
\usepackage{pgf,pgfarrows,pgfnodes,pgfautomata,pgfheaps}
\usepackage{colortbl}

\DeclareMathOperator{\D}{\operatorname{d}}

\newtheorem{theorem}{Theorem}[section]
\newtheorem{corollary}[theorem]{Corollary}
\newtheorem{proposition}[theorem]{Proposition}

\newtheorem{remark}[theorem]{Remark}

\def\irr#1{{\rm Irr}(#1)}
\def\irrr#1#2 {\irr {#1 \mid #2}}

\newcommand{\R}{\mathbb R}
\newcommand{\N}{\mathbb N}

\newcommand{\sfe}{{{\mathbb S}^{n-1}}}

\newcommand{\dist}{\mbox{\rm dist}}

\renewcommand{\d}{\mbox{\rm d}}

\newcommand{\rank}{\mbox{rank}}
\newcommand{\cl}{\mbox{\rm cl}}
\newcommand{\trace}{\mbox{tr}}

\newcommand{\xt}{\overline{x}}

\begin{document}

\title[]{\sc The Brunn-Minkowski inequality for the first eigenvalue\\
of the Ornstein-Uhlenbeck operator\\
and log-concavity of the relevant eigenfunction}
\author[Andrea Colesanti, Elisa Francini, Galyna Livshyts, Paolo Salani]{Andrea Colesanti, Elisa Francini, Galyna Livshyts, Paolo Salani}
\date{\today}
\begin{abstract} We prove that the first (nontrivial) Dirichlet eigenvalue of the Ornstein-Uhlenbeck operator
$$
L(u)=\Delta u-\langle\nabla u,x\rangle\,,
$$
as a function of the domain, is convex with respect to the Minkowski addition, and we characterize the equality cases in some classes of convex sets. We also prove that the corresponding (positive) eigenfunction is log-concave if the domain is convex.
\end{abstract} 
\maketitle

\section{Introduction}
In this paper we are concerned with the first Dirichlet eigenvalue $\lambda_\gamma$ of the Ornstein-Uhlenbeck operator, and the corresponding eigen\-function. Such eigenvalue represents the counterpart, in Gauss space, of the classical principal frequency, which can also be characterized as the optimal constant in the Poincar\'e inequality, or as the first Dirichlet eigenvalue of the Laplace operator.

We investigate in particular two aspects: a Brunn-Minkowski type inequality for $\lambda_\gamma$ with respect to the Minkowski addition, and the log-concavity of the corresponding (positive) eigen\-function. As we will see, the situation in the Gauss space closely parallels that of the classical setting, with some differences due to the lack of homogeneity and translation invariance of the Gauss measure.

\subsection{Short history in the classical Lebesgue case} For a bounded set $\Omega$ (with sufficiently regular boundary) in $\R^n$, the principal frequency of $\Omega$ is defined as follows:
$$
\lambda(\Omega):=\inf\left\{\frac{\displaystyle{\int_\Omega|\nabla v|^2 dx}}{\displaystyle{\int_\Omega v^2 dx}}\colon v\in W^{1,2}_0(\Omega),v\not\equiv0\right\}
$$   
(here and below we adopt the standard notation for Sobolev spaces). It is well-known, and follows via a variational argument, that $\lambda(\Omega)$ is the smallest positive number $\lambda$ such that there exists a solution of the boundary value problem:
\begin{equation}\label{EDP-lambda-Leb}
\left\{
\begin{array}{lll}
\Delta u=-\lambda u,\quad u>0,\quad\mbox{in $\Omega$,}\\
\\
u=0\quad\mbox{on $\partial\Omega$.}
\end{array}
\right.
\end{equation}
See e.g. \cite{Gilbarg-Trudinger,PolSz} for a classical exposition of this topic. 

The principal frequency $\lambda(\Omega)$ is monotone decreasing with respect to inclusion, that is, $\lambda(\Omega_1)\geq \lambda(\Omega_2)$ whenever $\Omega_1\subset \Omega_2$. Moreover $\lambda$ is a positively homogeneous functional of order $-2$:
$$
\lambda(t\Omega)=t^{-2}\lambda(\Omega)
$$
for any $t>0$. 

The principal frequency verifies various inequalities of isoperimetric type. The most significant is the Faber-Krahn inequality (see e.g. \cite{PolSz}) which states that
$$
\lambda(\Omega)\geq \lambda(\Omega'), 
$$
where $\Omega'$ is a ball having the same Lebesgue measure as $\Omega$ (an interesting strengthening of the Faber-Krahn inequality was proven by Kohler-Jobin \cite{KJ1978}, \cite{KJ1982}, see also Brasco \cite{Brasco}). 

\medskip

A different inequality, proven by Brascamp and Lieb \cite{Brascamp-Lieb-1976}, states that $\lambda^{-1/2}$ is concave with respect to the Minkowski addition of sets. We recall that for $A, B\subset \R^n,$ their Minkowski additon is defined as 
$$
A+B=\{x+y:\,x\in A,\, y\in B\}.
$$

\begin{theorem}[{\bf Brascamp-Lieb}]\label{BL thm} Let $\Omega_0,\Omega_1$ be connected, open and bounded subsets of $\R^n$, and let $t\in[0,1]$. Then 
\begin{equation}\label{concavity-lambda}
\lambda((1-t)\Omega_0+t\Omega_1)^{-\frac{1}{2}}\geq (1-t)\lambda(\Omega_0)^{-\frac{1}{2}}+ t\lambda(\Omega_1)^{-\frac{1}{2}}.
\end{equation}
\end{theorem}

We refer the reader to \cite{Borell-2000} for an alternative proof, and to \cite{Colesanti} for the characterization of equality conditions. In fact, in \cite{Brascamp-Lieb-1976} the following weaker inequality is proved:
\begin{equation}\label{concavity-lambda weak}
\lambda((1-t)\Omega_0+t\Omega_1)\leq (1-t)\lambda(\Omega_0)+ t\lambda(\Omega_1).
\end{equation}
On the other hand, \eqref{concavity-lambda} can be proved using \eqref{concavity-lambda weak} and a standard argument based on the homogeneity of $\lambda$. This aspect will mark a difference between the classical and the Gaussian case.  

\medskip

Inequality (\ref{concavity-lambda}) may be compared to the classical Brunn-Minkowski inequality: for every measu\-rable sets $A_0, A_1$ and $t\in[0,1]$, such that $(1-t)A_0+tA_1$ is measurable,
\begin{equation}\label{BrMi}
|(1-t)A_0+t A_1|^{\frac{1}{n}}\geq (1-t)|A_0|^{\frac{1}{n}}+t|A_1|^{\frac{1}{n}},
\end{equation}
where $|\cdot|$ denotes the Lebesgue measure. 
The analogy is strengthened by the fact that the exponent $1/n$ in \eqref{BrMi} is the reciprocal of the homogeneity order of the involved functional, like in \eqref{concavity-lambda}. We refer to \cite{Gardner,Schneider,AGMbook} for extensive treatises on the Brunn-Minkowski inequality.

\medskip

The concavity inequality (\ref{concavity-lambda}) is connected with a geometric property of the solution $u$ of \eqref{EDP-lambda-Leb}: if $\Omega$ is convex, then $u$ is log-concave. The log-concavity of $u$ was first proved in \cite{Brascamp-Lieb-1976}; different proofs can be found in \cite{Caffarelli-Friedman, Korevaar1,Korevaar-Lewis,Kennington} (see also \cite{Acker-Payne-Philippin} for a weaker version of this result). The link between the concavity of the variational functional $\lambda$ and the geometric property of the solution of the cor\-res\-ponding Euler-Lagrange equation is not evident,  but it is revealed by the proofs of the two facts. Indeed, more than one technique successfully exploited to prove the log-concavity of $u$ can be adapted to demonstrate the validity of \eqref{concavity-lambda} (see \cite{Brascamp-Lieb-1976, Korevaar1, Kennington}). This phenomenon can also be observed for other functionals, like the electrostatic capacity and the torsional rigidity (see for instance \cite{Colesanti}). As we will see in this paper, the eigenvalue of the Ornstein-Uhlenbeck operator provides a further example in this sense.

\subsection{The Gaussian eigenvalue problem}

We denote by $\gamma$ the Gaussian measure in $\R^n$, which is defined as follows: for a measurable subset $A$ of $\R^n$,
$$
\gamma(A)=\frac1{(2\pi)^{n/2}}\int_A e^{-|x|^2/2}\d x.
$$
In the sequel, $\d\gamma$ stands for integration with respect to $\gamma$.

Let $\Omega\subset\R^n$ be open and bounded and assume that $\partial\Omega$ is Lipschitz. The Gaussian version of the principal frequency for $\Omega$ can be defined as the solution of the following minimization problem:
\begin{equation}\label{lambdagamma}
\lambda_\gamma (\Omega)=\inf\left\{R_{\gamma}(v)\colon v\in W^{1,2}_0(\Omega,\gamma),v\not\equiv0\right\},
\end{equation}
where $R_{\gamma}$ is the Gaussian Rayleigh quotient, defined as
$$
R_{\gamma}(v)=\frac{\displaystyle{\int_\Omega|\nabla v|^2\d\gamma}}{\displaystyle{\int_\Omega v^2\d\gamma}}.
$$ 
A standard variational argument shows that the infimum in \eqref{lambdagamma} is in fact a minimum and $\lambda_\gamma(\Omega)$ coincides with the positive number $\lambda_\gamma$ such that the following boundary value problem has a solution:
\begin{equation}\label{EDP lambda}
\left\{
\begin{array}{lll}
Lu=-\lambda_\gamma u,\quad u>0,\quad\mbox{in $\Omega$,}\\
\\
u=0\quad\mbox{on $\partial\Omega$.}
\end{array}
\right.
\end{equation}
Here $L$ is the Ornstein-Uhlenbeck operator, namely, the second order differential operator defined on $C^2(\Omega)$ by the formula:
$$
Lu=\Delta u-\langle\nabla u,x\rangle.
$$
The solution of problem \eqref{EDP lambda} is unique, up to multiplication by positive constants (that is, $\lambda_\gamma$ is simple). 
Throughout the paper, we call {\em eigenfunction} of $\Omega$ any solution of \eqref{EDP lambda}. We mention that the definition of $\lambda_\gamma$ can be extended in a natural way to unbounded sets (see, for instance, \cite{Ehrhard-84,Liv}). 

We will refer to $\lambda_\gamma$ as to the first Dirichlet eigenvalue of the operator $L$ in $\Omega$, or as to the Gaussian principal frequency, and occasionally we indicate it as $\lambda_{\gamma}(\Omega)$ to emphasize its dependency on $\Omega$. We refer, e.g., to \cite{CarKer}, \cite{Liv}, \cite{HerLiv} for a general study of the Gaussian principal frequency in various contexts. Note that, similarly to the Lebesgue case, $\lambda_{\gamma}(\Omega_1)\geq \lambda_{\gamma}(\Omega_2)$ whenever $\Omega_1\subset\Omega_2$; however, the functional $\lambda_{\gamma}$ lacks homogeneity properties.

A Gaussian analogue of Faber-Krahn inequality was established by Ehrhard in \cite{Ehrhard-84} (see also \cite{CarKer} for a discussion on equality conditions). This result states for every open set $\Omega\subset\R^n$ 
$$
\lambda_{\gamma}(\Omega)\geq \lambda_{\gamma}(H)
$$
where $H$ is a half-space such that $\gamma(K)=\gamma(H)$ (for the Gaussian analogue of the Kohler-Jobin inequality, see \cite{HerLiv}). According to the Gaussian isoperimetric inequality (see \cite{Sudakov-Tsirelson-1974,Borell-1975}), half-spaces are the isoperimetric sets for the Gaussian measure; it is therefore not surprising that half-spaces also appear as minimizers for the Gaussian principal frequency (in fact, these kind of phenomena are directly related, see for instance \cite{CarKer}).

\medskip

In this paper we establish a convexity property for the Gaussian principal frequency $\lambda_\gamma$. In analogy with \eqref{concavity-lambda weak}, we show the following result. 

\begin{theorem}\label{thm BM for lambda} Let $\Omega_0, \Omega_1$ be open, bounded subsets of $\R^n$. Let $t\in[0,1]$ and set
$$
\Omega_t=(1-t)\Omega_0+t\Omega_1.
$$
Then 
\begin{equation}\label{BM for lambda}
\lambda_{\gamma}(\Omega_t)\le (1-t)\lambda_{\gamma}(\Omega_0)+t\lambda_{\gamma}(\Omega_1).
\end{equation}
\end{theorem}
As already observed, as $\lambda_\gamma$ is not homogeneous, \eqref{BM for lambda} does not imply a stronger inequality such as \eqref{concavity-lambda}. A similar fact happens for the volume: by the Pr\'ekopa-Leindler inequality (see \cite{Leindler,Prekopa-1971,Prekopa-1975}), for every measu\-rable sets $A_0, A_1$ and $t\in[0,1]$ such that $(1-t)A_0+tA_1$ is measurable, it holds
\begin{equation}\label{PL}
\gamma((1-t)A_0+t A_1)\ge\gamma(A_0)^{1-t}\gamma(A_1)^t.
\end{equation}
Inequality \eqref{PL} can be considered the Gaussian version of the Brunn-Minkowski inequality (see also \cite{Ehrhard-83, BorellEhrard} for the Ehrhard inequality, a stronger result). On the other hand, as it was first conjectured by Gardner and Zvavitch \cite{Gardner-Zvavitch}, and then proved by Eskenazis and Moschidis \cite{Eskenazis-Moschidis}, for convex sets $A_0$ and $A_1$ which are symmetric with respect to the origin, and for $t\in[0,1]$, it holds
$$
\gamma((1-t)A_0+tA_1)^{1/n}\ge
(1-t)\gamma(A_0)^{1/n}+t\gamma(A_1)^{1/n}
$$
(see also \cite{Kolesnikov-Livshyts} for important contributions to the conjecture by Gardner and Zvavitch). It would be interesting to understand whether a similar improvement holds for inequality \eqref{BM for lambda}.

\medskip

When restricting to suitably regular, symmetric with respect to the origin, convex sets, we also establish equality conditions for inequality \eqref{BM for lambda}.

\begin{theorem}\label{equality conditions} Let $\Omega_0, \Omega_1$ be open, bounded and convex subsets of $\R^n$.
Assume moreover that $\Omega_0$ and $\Omega_1$ are symmetric with respect to the origin, and $\partial\Omega_0$ and $\partial\Omega_1$ are of class $C^{2,\alpha}_+$, for some $\alpha\in(0,1)$. If, for some $t\in[0,1]$,
\begin{equation*}
\lambda_{\gamma}((1-t)\Omega_0+t\Omega_1)=(1-t)\lambda_{\gamma}(\Omega_0)+t\lambda_{\gamma}(\Omega_1),
\end{equation*}
then $\Omega_0=\Omega_1$.
\end{theorem}

Here $\partial\Omega\in C^{2,\alpha}_+$ means that the boundary of $\Omega$ is of class $C^{2,\alpha}$, and the Gauss curvature is strictly positive at every point of $\partial\Omega$.

\begin{remark} 
{\em In fact we treat equality conditions for a larger class of sets: convex sets such that the Laplacian of an eigenfunction is strictly positive. This is the content of Theorem \ref{thm BM for lambda restricted}, of which Theorem \ref{equality conditions} is a corollary. An intermediate result is Corollary \ref{cor1}, which refers to convex sets such that the maximum of an eigen\-function is at the origin. We prefer however to enlighten Theorem \ref{equality conditions} because here we have a clear and simple geometric conditions on the involved sets.}
\end{remark}

\medskip

Similarly to other functionals satisfying a Brunn-Minkowski type inequality (see for instance \cite{BS}, and \cite{S} for a related general theory), inequality \eqref{concavity-lambda} can be used to derive an Urysohn's type inequality for $\lambda$. This is a similar inequality to the previously mentioned Faber-Krahn inequality, where the volume is replaced by the {\em mean width}. We recall that, given an open bounded and convex set $\Omega$ and a direction $y\in\sfe$, the width of $\Omega$ in the direction $y$ is the distance between the supporting hyperplanes to $\Omega$ with outer unit normal $y$ and $-y$; we denote it by $w(\Omega,y)$. The mean width $w(\Omega)$ of $\Omega$ is then defined as
$$
w(\Omega)=\frac1{\mathcal{H}^{n-1}(\sfe)}\int_{\sfe}w(\Omega,y)\d{\mathcal H}^{n-1}(y),
$$ 
where ${\mathcal H}^{n-1}$ is the $(n-1)$-dimensional Hausdorff measure. Then, for every open bounded convex set $\Omega$, it holds
$$
\lambda(\Omega)\geq \lambda(\Omega'),
$$ 
where 
$$
\mbox{$\Omega'$ is any ball such that}\,\, w(\Omega')=w(\Omega).
$$

Similarly to the Lebesgue case, we get the following result.

\begin{corollary}\label{cor-min} Let $\Omega$ be an open, bounded and convex set in $\R^n$. Then
$$
\lambda_\gamma(\Omega)\geq \lambda_\gamma(\Omega^\sharp),
$$ 
where 
$$
\mbox{$\Omega^\sharp$ is a ball centered at $0$ such that}\,\, w(\Omega^\sharp)=w(\Omega).
$$
\end{corollary}

\begin{remark}
{\em Notice that, unlike the Lebesgue case, the ball $\Omega^\sharp$ must be centered at the origin (since the Gaussian measure is rotation invariant, but not translation invariant), and, more importantly, Corollary \ref{cor-min} does not follow from the Gaussian Faber-Krahn inequality. On one hand, the log-concavity and the rotation-invariance of the Gaussian measure imply (via a similar argument) that whenever $\Omega^\sharp$ is an origin symmetric ball such that $w(\Omega)=w(\Omega^\sharp)$, then $\gamma(\Omega)\geq \gamma(\Omega^\sharp)$. On the other hand, as we mentioned earlier, half-spaces are the minimizers in the Gaussian version of the Faber-Krahn inequality, rather than balls. So, the fact that the ball (rather than the half-space) is the extremizer in Corollary \ref{cor-min} sets it apart from the cohort of inherently Gaussian isoperimetric-type inequalities, in which the half-space is extremal. Such inequalities include, of course, the Gaussian isoperimetric inequality (see e.g. \cite{AGMbook2}), the Gaussian Faber-Krahn inequality (see \cite{Ehrhard-83,CarKer}), the Gaussian Saint-Venant inequality (see e.g. \cite{Ehrhard-83,Liv}), the Gaussian Kohler-Jobin from \cite{HerLiv}, and a variety of other inequalities such as the Gaussian barycenter inequality, the Gaussian dilation inequality and others (see \cite{Liv-notes}). Therefore, perhaps Theorem \ref{thm BM for lambda} is a property related to rotation invariance of the measure rather than to its special Gaussian properties.}
\end{remark}

\subsection{Log-concavity of the solution}

Theorem \ref{thm BM for lambda} is deeply connected to a geometric property of the first eigen\-function of the operator $L$, like in the Lebesgue case. This is the content of the second main result of this paper.

\begin{theorem}\label{thm log-concavity lambda 3}
Let $\Omega$ be an open, bounded and convex subset of $\R^n$. Let $u$ be an eigenfunction of $\Omega$, that is, a solution of problem \eqref{EDP lambda} in $\Omega$. Then the function 
$$
W=-\ln(u)
$$
is convex in $\Omega$. 
\end{theorem}

\subsection{The strategy of proofs and connections between Theorem \ref{thm BM for lambda} and Theorem \ref{thm log-concavity lambda 3}}

We start noticing that several techniques, developed to prove convexity type properties of solutions to elliptic boundary value problems, could be adapted to prove Theorems \ref{thm BM for lambda} and \ref{thm log-concavity lambda 3} (see Remark \ref{otherproofs}).

Moreover, Theorem \ref{thm log-concavity lambda 3} could be obtained as a consequence of \cite[Lemma 4]{BS}. Moreover, the Brunn-Minkowski inequality for $\lambda_\gamma$, in the case of convex sets, can be deduced as an application of \cite[Lemma 5.1]{S} (where comparison principle is not needed). To obtain Theorem \ref{thm BM for lambda} without the assumption of convexity, we use the technique of \cite{S} (where indeed convexity of the domains, as well as power concavity properties of the involved functions are not really needed in most of the results). Then, for the convenience of the reader, we will give also a complete proof of Theorem \ref{thm log-concavity lambda 3} as a suitable modification of the one of Theorem  \ref{thm BM for lambda}, to enlighten the strict relation among the two results.

Although we prove inequality \eqref{BM for lambda} for open bounded sets, we are able to establish its equality conditions only under more restrictive assumptions on the involved domains: central symmetry, regularity and strong convexity of the boundary. These assumptions permit to prove (via the constant rank method) a stronger version of Theorem \ref{thm log-concavity lambda 3}, namely that the logarithms of the involved eigen\-functions have everywhere positive definite Hessian matrices. This property allows to exploit effe\-ctively their Legendre-Fencel transforms. However, we believe that it should be possible to treat the equality case of the Brunn-Minkowski inequality for the Gaussian eigenvalue in the general case and we state the following conjecture.

\medskip

\noindent {\bf Conjecture.} {\em Equality holds in \eqref{BM for lambda} if and only if $\Omega_0$ and $\Omega_1$ coincide and are convex, up to removing null measure sets.}

\medskip

\noindent{\bf Acknowledgements.} The first, second and last authors are partially supported by INdAM through different GNAMPA projects, and by the Italian "Ministero dell'Universit\`a e Ricerca" and EU through different PRIN 2022 projects, within the Next Generation EU program. The third author was supported by the National Science Foundation grant NSF-BSF DMS-2247834.

\section{Preliminaries}

We work in the $n$-dimensional Euclidean space $\R^n$. We denote by $|\cdot|$ and $\langle\cdot,\cdot\rangle$ the Euclidean norm and the standard scalar product in $\R^n$, respectively. Given a subset $A$ of $\R^n$, we denote by $\cl(A)$ the closure of $A$. 

By $|\cdot |$ we denote the Lebesgue volume and $B^n_2$ is the unit Euclidean ball.

\subsection{The Ornstein-Uhlenbeck operator}
Let $\Omega$ be an open subset of $\R^n$, and $u\in C^2(\Omega)$. We set
$$
Lu(x)=\Delta u(x)-\langle\nabla u(x),x\rangle.
$$
We will refer to $L$ as to the Ornstein-Uhlenbeck operator (in the one dimensional case $L$ is involved also in the Hermite differential equation). 

Integration by parts shows that $L$ is self-adjoint with respect to the Gaussian measure: if $\Omega$ is an open subset of $\R^n$, and $\phi,\psi\in C^{\infty}_c(\Omega)$, then 
$$
\int_\Omega\varphi L\psi\D\gamma=-\int_\Omega\langle\nabla\varphi,\nabla\psi\rangle\D\gamma=
\int_\Omega\psi L\varphi\D\gamma.
$$

\subsection{Existence of the eigenvalue}

The following existence result is classical, see e.g. \cite{Gilbarg-Trudinger,Evans}.

\begin{theorem} Let $\Omega$ be an open and bounded subset of $\R^n$, with Lipschitz boundary. Then there exists a unique number $\lambda_\gamma>0$, such that the following problem
\begin{equation}\label{EDP lambda bis}
\left\{
\begin{array}{lll}
Lu=-\lambda_\gamma u,\quad u>0,\quad\mbox{in $\Omega$,}\\
\\
u=0\quad\mbox{on $\partial\Omega$,}
\end{array}
\right.
\end{equation}
admits a solution $u$ (which we call an {\em eigenfunction} of $\Omega$). The solution is unique up to multiplication by positive constants. 
\end{theorem}

As mentioned in the introduction, $\lambda_\gamma$ can be characterized as follows:
$$
\lambda_\gamma=\inf\frac{\displaystyle{\int_\Omega|\nabla v|^2\D\gamma}}{\displaystyle{\int_\Omega v^2\D\gamma}}
$$ 
where the infimum is taken over all functions $v\in W^{1,2}(\Omega,\gamma)$ such that 
$$
\int_\Omega v^2\D\gamma>0. 
$$
In particular, any solution of \eqref{EDP lambda bis} is a minimizer of the previous problem and viceversa.

\section{The proofs of Theorems \ref{thm BM for lambda} and \ref{thm log-concavity lambda 3}}
\label{section viscosity}

We start recalling some notions from the theory of viscosity solutions, for which we refer the reader to \cite{userguide}.

Let $u,\varphi\colon\Omega\to\R$, where $\Omega$ is an open subset of $\R^n$. We say that $\varphi$ {\em touches $u$ from above at a point $\xt\in\Omega$} if $\varphi(\xt)=u(\xt)$ and $\varphi\geq u$ in a neighborhood of $\xt$. Similarly, $\varphi$ {\em touches $u$ from below at $\xt$} if $\varphi(\xt)=u(\xt)$ and $\varphi\leq u$ in a neighborhood of $\xt$.

Notice that if $\varphi$ and $\psi$ are twice differentiable at $\xt$ and $\varphi$ touches $\psi$ from above at $\xt$, then
\begin{equation}\label{useful}
\varphi(\xt)=\psi(\xt)\,,\quad \nabla\varphi(\xt)=\nabla\psi(\xt)\,,\quad D^2\varphi(\xt)\geq D^2\psi(\xt)\,.
\end{equation}

An upper continuous function $v$ is a {\em viscosity subsolution} of the equation
$$
\Delta u - \langle x,\nabla u\rangle+\lambda u=0
$$
in $\Omega$, if for every point $\xt\in\Omega$ and every test function $\varphi$ touching $u$ from above at $\xt$, it holds
$$
\Delta\varphi(\xt)-\langle\xt,\nabla\varphi(\xt)\rangle+\lambda\varphi(\xt)\geq 0\,.
$$

\medskip

In fact, we are here interested in {\em distributional solutions} of \eqref{EDP lambda bis}. Luckily, viscosity solutions and distributional solutions coincide in this case by a famous result of Ishii \cite{Ishii}, so allowing us to integrate by parts when needed, as we will see.

\begin{remark}
{\em The definition  of  viscosity solutions is based on the maximum principle (and especially on the fact that if $v\in C^{2}(\Omega)$ attains its maximum at $ x\in\Omega$, then $\nabla v(x)=0$ and $D^{2}v(x)\leq0$) and it is particularly useful when a comparison principle holds. On the other hand, here we are treating an eigenvalue problem, in which no comparison principle can work since the solution is insensitive to positive scalar multiplication, that is, if $u$ is a solution of problem \eqref{EDP lambda bis}, then $au$ is a solution too for any $a>0$.}
\end{remark}

\begin{proof}[Proof of Theorem \ref{thm BM for lambda}.] The proof largely follows the steps of \cite[Lemma 5.1]{S}, which in turn ge\-neralizes ideas introduced in \cite{MBS} and strongly inspired by \cite{Alvarez-Lasry-Lions}.

Let $\Omega_0$ and $\Omega_1$ be open and bounded. Let $u_i$ be the solution of problem \eqref{EDP lambda}, for $i=0,1$. For simplicity, set
$$
\lambda_i=\lambda_{\gamma}(\Omega_i)\,\text{ for }i=0,1\,,\quad \lambda_t=(1-t)\lambda_0+t\lambda_1\,.
$$
Moreover, for $x\in\Omega_t$ let
\begin{equation}\label{ut}
u_t(x)=\max\{u_0(x_0)^{1-t}u_1(x_1)^t\,:\,x_i\in\overline{\Omega}_i\,\text{for }i=0,1,\; x=(1-t)x_0+tx_1\}.
\end{equation}
Notice that $u_t\in C(\cl(\Omega_t))$ and that
$$
u_t>0\,\text{ in }\Omega_t,\quad u_t=0\,\text{ on }\partial\Omega_t,
$$
as $u_i\in C(\cl(\Omega_i))$, $u_i>0$ in $\Omega_i$ and $u_i=0$ on $\partial\Omega_i$ for $i=0,1$.
Furthermore, for every $\xt\in\cl(\Omega_t)$ there exist $\xt_0\in\cl(\Omega_0),\,\xt_1\in\cl(\Omega_1)$ where the maximum in the definition of \eqref{ut} is attained, that is 
\begin{equation}\label{x0x1}
\xt=(1-t)\xt_0+t\xt_1\,,\quad u_t(\xt)=u_0(\xt_0)^{1-t}u_1(\xt_1)^t.
\end{equation}
Notice that if $\xt\in\partial\Omega_t$, then clearly $\xt_i\in\partial\Omega_i$ for $i=0,1$, while if $\xt\in\Omega_t$, then $\xt_i\in\Omega_i$ for $i=0,1$ since $u_t(\xt)>0$. In the latter case, as $u_0,u_1$ are differentiable at $x_0,x_1$ respectively, we can also easily see that
\begin{equation}\label{nablatheta}
\frac{\nabla u_0(\xt_0)}{u_0(\xt_0)}=\frac{\nabla u_1(\xt_1)}{u_1(\xt_1)}:=\theta.
\end{equation}

Now we want to prove that $u_t$ is a viscosity subsolution of problem \eqref{EDP lambda} when $\Omega=\Omega_t$ and $\lambda=\lambda_t$. 
For showing this, we construct, for every point $\xt\in\Omega_t$, a $C^2$ function $\psi$ touching $u_t$ from below at $\xt$, and such that
$$
\Delta\psi(\xt)-\langle\xt,\nabla\psi(\xt)\rangle+\lambda_t\psi(\xt)=0\,.
$$
Indeed, this yields that every test function $\phi$ touching $u_t$ from above at $\xt$ must also touch $\psi$ from above at $\xt$ and \eqref{useful} gives 
$$
\Delta\phi(\xt)-\langle\xt,\nabla\phi(\xt)\rangle+\lambda_t\phi(\xt)\geq0\,,
$$
as desired. In fact, we need to construct $\psi$ only for the points $\xt$ such that $u_t(\xt)>u(\xt)$, indeed, if $u_t(\xt)=u(\xt)$, every test function touching $u_t$ from above at $\xt$, also touches $u$ from above at $\xt$ and the latter inequality is straightforward. Then, let $\xt\in\Omega_t$ be such that $u_t(\xt)>u(\xt)$ and let $\xt_0\in\Omega_0$ and $\xt_1\in\Omega_1$ be the points given by \eqref{x0x1}. In a sufficiently small neighborhood $B$ of $\xt$, we define
$$
\psi(x)=u_0(x-\xt+\xt_0)^{1-t}u_1(x-\xt+x_1)^t\,.
$$
Then, by the choice of $\xt_0$ and $\xt_1$ and the definition of $u_t$, it is easily seen that $\psi$ touches $u_t$ from below at $\xt$. Furthermore $\psi\in C^2(B)$ and we can calculate
$$
\nabla\psi(x)=\psi(x)\left[(1-t)\frac{\nabla u_0(x-\xt+\xt_0)}{u_0(x-\xt+\xt_0)}+t\frac{\nabla u_1(x-\xt+\xt_1)}{u_1(x-\xt+\xt_1)}\right]
$$
and
$$
\begin{array}{ll}D^2\psi(x)=&\dfrac{\nabla\psi(x)\times\nabla\psi(x)}{\psi(x)}+\\
\\
&-\psi(x)\left[(1-t)\dfrac{\nabla u_0(x-\xt+\xt_0)\times\nabla u_0(x-\xt+\xt_0)}{u_0(x-\xt+\xt_0)^2}\right.\\
\\
&\left.+t\dfrac{\nabla u_1(x-\xt+\xt_1)\times\nabla u_1(x-\xt+\xt_1)}{u_1(x-\xt+\xt_1)^2}\right]+\\
\\
&+\psi(x)\left[(1-t)\dfrac{D^2 u_0(x-\xt+\xt_0)}{u_0(x-\xt+\xt_0)}+t\dfrac{D^2 u_1(x-\xt+\xt_1)}{u_1(x-\xt+\xt_1)}\right],
\end{array}
$$
whence
\begin{equation}\label{nablapsi}
\nabla\psi(\xt)=\psi(\xt)\,\theta
\end{equation}
and
\begin{equation}\label{D2psi}
D^2\psi(\xt)=\psi(\xt)\left[(1-t)\dfrac{D^2 u_0(\xt_0)}{u_0(\xt_0)}+t\dfrac{D^2 u_1(\xt_1)}{u_1(\xt_1)}\right]\,,
\end{equation}
where in the former we have used \eqref{nablatheta} and in the latter we have used the former and again \eqref{nablatheta}.
From \eqref{nablapsi} and \eqref{D2psi} we finally get
\begin{equation}\label{deltapsi}
\begin{array}{ll}
\Delta\psi(\xt)&=\psi(\xt)\left[(1-t)\dfrac{\Delta u_0(\xt_0)}{u_0(\xt_0)}+t\dfrac{\Delta u_1(\xt_1)}{u_1(\xt_1)}\right]\\
\\
&=\psi(\xt)\left[(1-t)\dfrac{\langle\xt_0,\nabla u_0(\xt_0)\rangle-\lambda_0u_0(\xt_0)}{u_0(\xt_0)}+t\dfrac{\langle\xt_1,\nabla u_1(\xt_1)\rangle-\lambda_1u_1(\xt_1)}{u_1(\xt_1)}\right]\\
\\
&=\psi(\xt)\left[(1-t)\left(\langle\xt_0,\dfrac{\nabla u_0(\xt_0)}{u_0(\xt_0)}\rangle-\lambda_0\right)+t\left(\langle\xt_1,\dfrac{\nabla u_1(\xt_1)}{u_1(\xt_1)}\langle-\lambda_1\right)\right]\\
\\
&=\psi(\xt)\left[\langle(1-t)\xt_0+t\xt_1,\theta\rangle-\left((1-t)\lambda_0+t\lambda_1\right)\right]\\
\\
&=\langle\xt,\nabla\psi(\xt)\rangle-\lambda_t\psi(\xt).
\end{array}
\end{equation}
In the above series of equalities we have made use of the equations satisfied by $u_0$ and $u_1$, of \eqref{nablatheta} and of \eqref{nablapsi}.

So far, we proved that $u_t$ verifies the inequality
\begin{equation}\label{viscosity inequality ut}
\Delta u_t(x)-\langle\nabla u_t(x), x\rangle\ge-\lambda_t u_t\quad\text{in }\Omega_t
\end{equation}
in the viscosity sense, and consequently also in the distributional sense, see \cite[Theorem 1]{Ishii}.

Now multiply \eqref{viscosity inequality ut} by $-u_t(x)$ to get
$$
-u_t\Delta u_t(x)+u_t\langle\nabla u_t(x), x\rangle\le\lambda_tu_t^2\,,
$$
and integrate in $d\gamma$ on $\Omega_t$.  
Then, taking into account that
$$
-e^{-|x|^2/2}u_t\Delta u_t+u_t\langle\nabla u_t(x), x\rangle e^{-|x|^2/2}= -\text{div}(e^{-|x|^2/2}u_t\nabla u_t)+|\nabla u_t|^2e^{-|x|^2/2}
$$
and that $u_t$ vanishes on $\partial\Omega_t$, an integration by parts  gives
\begin{equation}\label{Ru_t}
R_\gamma(u_t)\leq\lambda_t\,,
\end{equation}
whence \eqref{BM for lambda} follows by the very definition of $\lambda_\gamma(\Omega_t)$.
\end{proof}

\begin{proof}[Proof of Theorem \ref{thm log-concavity lambda 3}.] The proof could obtained by applying \cite[Theorem 1]{BS}. However, it is useful to give here a complete proof to enlighten the relation between the Brunn-Minkowski inequality of the eigenvalue and the log-concavity property of the eigenfunction (for this, see also \cite{S}). We repeat the argument of the previous proof with $\Omega_0=\Omega_1=\Omega$, and consequently $u_0=u_1=u$ and $\lambda_t=\lambda_\gamma(\Omega)$. For $t\in[0,1]$, the function $u_t$ is defined similarly as in \eqref{ut}:
$$
u_t(x)=\max\{u(x_0)^{1-t}u(x_1)^t\colon x_0, x_1\in\Omega,\ (1-t)x_0+tx_1=x\}.
$$
Notice that the {\it log-concave envelope} $u^*$ of $u$, i.e. the smallest log-concave function greater than or equal to $u$ in $\Omega$, is given by
$$
u^*(x)=\sup\{u_t(x)\,:\, t\in[0,1]\}\,,
$$
and that $u\leq u_t\leq u^*$ in $\Omega$ for every $t\in[0,1]$, while
$$
\text{$u$ is log-concave (i.e. it coincides with $u^*$)  if and only if $u=u_t$ for every $t\in[0,1]$.}
$$ 

With the same steps as before, it can be proved that $u_t$ verifies, in the viscosity (and distributional) sense, the following inequality:
\begin{equation}\label{viscosity inequality ut 2}
\Delta u_t(x)-\langle\nabla u_t(x), x\rangle+\lambda_\gamma(\Omega)u_t\ge0
\end{equation}
in $\Omega$. Now we multiply by $-u_t$ and integrate by parts in $d\gamma$, as in the proof of Theorem \ref{thm BM for lambda}, to get 
$$
R_\gamma(u_t)\leq\lambda_\gamma(\Omega)\,,
$$
whence indeed $R_\gamma(u_t)=\lambda_\gamma(\Omega)$ and $u_t$ is an eigenfunction, then $u_t=au$ for some $a>0$. On the other hand, by the very definition if $u_t$, we have
$$
\max_{\overline\Omega}u_t=\max_{\overline\Omega}u\,,
$$ 
whence $a=1$ and $u_t$ coincides with $u$ in $\Omega$, as desired.
\end{proof}

\begin{remark}\label{otherproofs} {\em Under suitable assumptions, the log-concavity of the eigenfunction of a convex domain can be deduced also in several other ways, for instance by means of the concavity function (first developed by Korevaar \cite{Korevaar1} and Kennington \cite{Kennington}) or via the constant rank method (see \cite{Acker-Payne-Philippin,Caffarelli-Friedman,Korevaar-Lewis}, and \cite{Guan-Xu} for a survey on more recent developments of this method). We would also like to thank E. Milman \cite{Milman} for showing to us a further elegant proof based on a clever transform from the Ornstein-Uhlenbeck to the Schr\"odinger operator and on \cite[Theorem 6.1]{Brascamp-Lieb-1976}.}
\end{remark}

\begin{proof}[Proof of Corollary \ref{cor-min}.] The proof closely follows the steps of \cite[Corollary 2.2]{BS} and \cite[Section 6]{S}. The mean width is linear with respect to the Minkowski addition  and clearly invariant with respect to rigid motions. Hence, given a sequence of rotations (about the origin) $\{\rho_j\}_{j\in\N}\subset SO(n)$, we define the {\em rotation mean} $\Omega^\sharp_m$ of $\Omega$ as follows
$$
\Omega^\sharp_m=\frac{1}{m}\left(\rho_1\Omega+\dots+\rho_m\Omega\right)\,.
$$
and we get that 
$$
w(\Omega^\sharp_m)=w(\Omega)\,\,\text{for every }m\in\N\,.
$$
Moreover, the rotation-invariance of the Gaussian measure yields
\begin{equation}\label{rot-inv-property}
\lambda_\gamma(\rho_j\Omega)=\lambda_\gamma(\Omega)\,\,\text{for every }j\in\N\,,
\end{equation}
Then, inequality (\ref{BM for lambda}) combined with (\ref{rot-inv-property}) shows that 
$$
\lambda(\Omega^\sharp_m)\,\,\text{is a decreasing sequence}\,.
$$
On the other hand, by a theorem of Hadwiger, see \cite[Theorem 3.3.2]{OldSchneider}, there exists a sequence of rotation means of $\Omega$ converging in Hausdorff metric to a ball $\Omega^\sharp$, centered at the origin, with diameter equal to the mean width $w(\Omega)$ of $\Omega$. Finally, the conclusion follows, in view of the fact that $\lambda_{\gamma}$ is continuous in Hausdorff metric.
\end{proof}

\section{Strong log-concavity of the eigenfunction}\label{section constant rank}

In this section we see how, under additional symmetry and regularity assum\-ptions on $\Omega$, we can improve the conclusion of Theorem \ref{thm log-concavity lambda 3}, in the following sense: the Hessian of $-\ln(u)$ is positive definite in $\Omega$. This will allow us to investigate equality conditions of \eqref{BM for lambda} in some restricted case and then to prove Theorem \ref{equality conditions} in the final section of the paper.

\medskip
We say that the boundary of an open set $\Omega\subset\R^n$ is of class $C^{2,\alpha}_+$, for some $\alpha\in(0,1)$, if it is of class $C^{2,\alpha}$ and the Gauss curvature is strictly positive at every point of $\partial\Omega$. 

\begin{theorem}\label{thm log-concavity lambda}
Let $\Omega$ be an open, bounded and convex subset of $\R^n$, with boundary of class $C^{2,\alpha}_+$, for some $\alpha\in(0,1)$, and let $u$ be a solution of problem \eqref{EDP lambda} in $\Omega$. Assume that 
\begin{equation}\label{Deltamin0}
\Delta u<0\,\,\text{ in }\Omega\,.
\end{equation}
Then the function 
$$
W=-\ln(u)
$$
is {\em strongly convex} in $\Omega$, i.e.,
\begin{equation}\label{maximal rank}
D^2 W(x)>0\quad\forall\, x\in\Omega.
\end{equation}
\end{theorem}

The proof is based on the so called constant rank method, introduced in \cite{Acker-Payne-Philippin} and subsequently developed in \cite{Caffarelli-Friedman}, \cite{Korevaar-Lewis}. We also point out that recently this approach led to new important results in the theory of geometric properties of solutions to elliptic and parabolic PDE's. We refer the reader to recent papers about this subject, like \cite{Ma-Ou,Guan-Xu} and therein references.  

\medskip

Before proving Theorem \ref{thm log-concavity lambda} (whose proof will be indeed a straightforward application of Theorem \ref{thm log-concavity lambda 3} and two technical results, namely Proposition \ref{boundary behavior} and Proposition \ref{teo rango costante}, to be proven hereafter), let us see two simple, but interesting, corollaries.
\begin{corollary}\label{max0}
Let $\Omega$ be an open, bounded and convex subset of $\R^n$, with boundary of class $C^{2,\alpha}_+$, for some $\alpha\in(0,1)$, and let $u$ be a solution of problem \eqref{EDP lambda} in $\Omega$. Assume that $u(0)=\max_{\overline\Omega}u$. Then  the function 
$$
W=-\ln(u)
$$
is {\em strongly convex} in $\Omega$.
\end{corollary}
\begin{proof}
By Theorem \ref{thm log-concavity lambda 3}, $u$ is log-concave, then its level sets are all convex. Since $0$ is a maximum point, it is contained in every superlevel set, which is then starshaped with respect to $0$. This implies $\langle x,\nabla u\rangle\leq 0$ in $\Omega$, whence $\Delta u<0$ and we can apply Theorem \ref{thm log-concavity lambda}.
\end{proof}
\begin{corollary}\label{symmetric}
Let $\Omega$ be an open, bounded and convex subset of $\R^n$, with boundary of class $C^{2,\alpha}_+$, for some $\alpha\in(0,1)$, and let $u$ be a solution of problem \eqref{EDP lambda} in $\Omega$. Assume furthermore $\Omega$ be symmetric with respect to the origin. Then  the function 
$$
W=-\ln(u)
$$
is {\em strongly convex} in $\Omega$.
\end{corollary}
\begin{proof}
The symmetry of $\Omega$ easily implies hat $u$ is also symmetric with respect to the origin, which is then a maximum point of $u$ and we can apply the previous corollary.
\end{proof}

\subsection{Boundary behaviour}

We will first show a result on the boundary behavior, that will be used in the proof of the above theorem as a companion of the constant rank theorem.

Given $X=(X_1,\dots,X_n), Y=(Y_1,\dots,Y_n)\in\R^n$, we denote by $X\otimes Y$ the $n\times n$ matrix with $ij$-th entry equal to $X_iY_j$, for $i,j=1,\dots,n$. Moreover, if $M$ is a square symmetric matrix, the notation $M>0$ means that $M$ is positive definite. 

\begin{proposition}\label{boundary behavior} Let $\Omega$ be an open, bounded and convex subset of $\R^n$; assume that $\partial\Omega$ is of class $C^{3,\alpha}$, $\alpha\in(0,1)$, and the Gauss curvature is strictly positive at each point of $\partial\Omega$. Let $w\in C^3(\bar\Omega)$ be such that 
$$
w>0\;\mbox{in $\Omega$,}\quad w=0\;\mbox{on $\partial\Omega$,}\quad\frac{\partial w}{\partial\nu}<0\;\mbox{on $\partial\Omega$,}
$$
where $\nu$ denotes the outer unit normal to $\partial\Omega$. Then there exists $\varepsilon>0$ such that
\begin{equation}\label{bb1}
\nabla w(x)\otimes\nabla w(x)-w(x)D^2 w(x)>0,
\end{equation}
for every $x\in\bar\Omega$ such that $\dist(x,\partial\Omega)<\varepsilon$.
In particular, setting
$$
W=-\ln(w).
$$
we have 
$$
D^2W(x)>0,\quad\forall\, x\in\{y\in\Omega\colon\dist(y,\partial\Omega)<\varepsilon\}.
$$
\end{proposition}

\begin{proof} As 
$$
\frac{\partial w}{\partial\nu}\ne0
$$
on $\partial\Omega$, by continuity there exists $\varepsilon_0>0$ such that 
$$
\frac m2\le|\nabla w(x)|\le 2M
$$ 
for every $x$ which verifies $\dist(x,\partial\Omega)\le\epsilon_0$. 

Let $x\in\bar\Omega$ be such that $\nabla w(x)\ne0$ and $t=w(x)$; we denote by $\kappa_i(x)$, $i=1,\ldots,(n-1)$ the principal curvatures of the level set $\{w=t\}$ at the point $x$. In particular, if $x\in\partial\Omega$, these are the principal curvatures of $\partial\Omega$ at $x$. Let $\eta>0$ be such that $\kappa_i(x)\ge\eta$ for every $i=1,\dots,(n-1)$, and for every $x\in\partial\Omega$. Up to reducing the value of $\varepsilon_0$ we may assume that
$$
\kappa_i(x)\ge\frac\eta2,
$$ 
for every $x$ such that $\dist(x,\partial\Omega)\le\epsilon_0$ and for every $i=1,\ldots,n-1$. Let 
$$
\Omega_0:=\{x\in\bar\Omega\colon 0\le\dist(x,\partial\Omega)\le\varepsilon_0\}.
$$

For $k=1,\ldots,n-1$, we denote by $\sigma_k(x)$ the $k$-th elementary symmetric function of the numbers $\kappa_i(x)$:
$$
\sigma_k(x)=\sum_{1\le i_1<i_2\dots<i_k\le n-1}\kappa_{i_1}\kappa_{i_2}\cdots\kappa_{i_k}.
$$ 
For completeness, we set $\sigma_0(x)=1$, for every $x$. For $x\in\bar\Omega$, and for $k\in\{1,\ldots,n\}$, let
$$
S_k(D^2 W(x))
$$ 
be the $k$-th elementary symmetric function of the eigenvalues of $D^2 W(x)$, with $S_0(D^2 W(x))=1$ for every $x$. We are going to prove that, for a suitable $\varepsilon>0$, $\varepsilon\le\varepsilon_0$, 
\begin{equation}\label{ksf positive}
S_k(D^2W(x))>0,
\end{equation} 
for every $k=1,\ldots,n$, and for every $x\in\bar\Omega$ such that $\dist(x,\partial\Omega)\le\varepsilon$. Condition \eqref{ksf positive} is equivalent to
$$
D^2 W(x)>0,
$$ 
hence this will conclude the proof. 

Let us fix $x_0\in\Omega_0$. We may choose a coordinate system such that the first $(n-1)$ coordinates are parallel to the directions of the principal curvatures of the level set $\{W=W(x_0)\}$, and the $n$-th coordinate is parallel to $\nabla W(x_0)$. With respect to this system we have 
\medskip
$$
D^2W(x_0)=
\left(
\begin{array}{cccccc}
\kappa_1(x_0)|\nabla W(x_0)|&0&0&\ldots&0&W_{1n}(x_0)\\
\\
0&\kappa_2(x_0)|\nabla W(x_0)|&0&\ldots&0&W_{2n}(x_0)\\
\\
\vdots&\vdots&\vdots&\ddots&\vdots&\vdots\\
\\
0&0&0&\ldots&\kappa_{n-1}(x_0)|\nabla W(x_0)|&W_{(n-1)n}(x_0)\\
\\
W_{n1}(x_0)&W_{n2}(x_0)&W_{n3(x_0)}&\ldots&W_{n(n-1)}(x_0)&W_{nn}(x_0)
\end{array}
\right)
$$
\medskip
Moreover:
$$
|\nabla W(x_0)|=\frac{|\nabla w(x_0)|}{w(x_0)},\quad W_{nj}(x_0)=W_{jn}(x_0)=-\frac{w_{nj}(x_0)}{w(x_0)},\quad W_{nn}(x_0)=\frac{|\nabla w(x_0)|^2}{w^2(x_0)}-\frac{w_{nn}(x_0)}{w(x_0)}.
$$
The quantity $S_k(D^2W(x_0))$ can be expressed as follows:
\begin{equation}\label{Sk}
S_k(D^2W(x_0))=\frac1k\sum_{1\le i_1<\ldots<i_k\le n,\;1\le j_1<\ldots<j_k\le n}\delta
\left(
\begin{array}{cccc}
i_1&i_2&\ldots&i_k\\
j_1&j_2&\ldots&j_k
\end{array}
\right)
W_{i_1j_1}\dots W_{i_kj_k},
\end{equation}
where the symbol
$$
\delta
\left(
\begin{array}{cccc}
i_1&i_2&\ldots&i_k\\
j_1&j_2&\ldots&j_k
\end{array}
\right)
$$
is equal to $1$ (respectively, to $-1$), if $(i_1,\dots,i_k)$ is an even (respectively, odd) permutation of $(j_1,\dots,j_k)$, and it is $0$ otherwise.

We write 
\begin{equation*}
S_k(D^2W(x_0))=A+B
\end{equation*}
where $A$ contains all the terms of the sum in \eqref{Sk} such that $i_k=j_k=n$, and $B$ all the other terms. Then
\begin{eqnarray*}
A&=&\frac1k |\nabla W(x_0)|^{k-1}W_{nn}(x_0)\sigma_{k-1}(x_0)\\
&=&\frac{|\nabla w(x_0)|^{k-1}}{kw(x_0)^{k+1}}(|\nabla w(x_0)|^2-w(x_0)w_{nn}(x_0))\sigma_{k-1}(x_0).
\end{eqnarray*}
On the other hand
$$
B=\frac1{w(x_0)^{k}}F(x_0)
$$
where $F=F(x)$, $x\in\Omega_0$, is the sum of products of first and second derivatives of $w$. In particular there exists a constant $N>0$ such that 
$$
|F(x)|\le N
$$
for every $x\in\Omega_0$. Hence, in an arbitrary point $x\in\Omega_0$, we have
$$
S_k(D^2W(x))=\frac{1}{kw(x)^{k+1}}\left[|\nabla w(x)|^{k+1}\sigma_{k-1}(x)+kw(x)F(x)\right].
$$
Recall that, in $\Omega_0$:
$$
|\nabla w(x)|^{k+1}\sigma_{k-1}(x)\ge c m^{k+1}\eta^{k-1}
$$
for some dimensional constant $c>0$, where $m>0$ is such that $|\nabla u(x)|\ge m$ for every $x\in\partial\Omega$. As 
$$
\lim_{x\to\partial\Omega} w(x)=0,
$$
there exists $\varepsilon\in(0,\varepsilon_0)$ such that 
$$
0\le w(x)\le\frac{c m^{k+1}\eta^{k-1}}{4N}
$$
for every $x$ such that $\dist(x,\partial\Omega)\le\varepsilon$. We have proved that
$$
S_k(D^2W(x))>0
$$
for every $x$ such that $\dist(x,\partial\Omega)\le\varepsilon$.
\end{proof}

\subsection{The constant rank theorem}
The following result is a slight generalization of the constant rank theorem proved by Korevaar and Lewis \cite[Theorem 1]{Korevaar-Lewis}.
\begin{proposition}\label{teo rango costante} Let $F\colon\R\times\R^n\to\R$ be a function of class $C^2(\R\times\R^n)$. Assume that
$$
F_{tt}(t,y)\le0\quad\forall (t,y)\in\R\times\R^n.
$$
Let $\Omega$ be an open, bounded and convex subset of $\R^n$. Let $v\in C^4(\Omega)$ be a solution of the equation
\begin{equation}\label{EDP per v}
\Delta v=F(v,\nabla v)+\langle\nabla v(x),x\rangle\quad\mbox{in $\Omega$.}
\end{equation}
Assume moreover that
$$
\Delta v>0\quad\mbox{in $\Omega$.}
$$
If $v$ is convex in $\Omega$, then $D^2 v$ has constant rank in $\Omega$. 
\end{proposition}

\begin{proof} We closely follow the scheme of \cite{Korevaar-Lewis}. Let $x_0$ be a point where the rank of $D^2 v$ is minimum. If $r=n$, there is nothing to prove. Hence, for the rest of the proof we assume that
$$
\rank(D^2 v(x_0))=r<n.
$$
For $x\in\Omega$, we denote by $\varphi(x)$ the elementary symmetric function of the eigenvalues of $D^2v(x)$, of order $(r+1)$. Then
$$
\varphi(x)\geq 0\mbox{ in }\Omega \mbox{ and }\varphi(x_0)=0.
$$
We will prove that there exists positive constants $c_1$ and $c_2$, such that
\begin{equation}\label{PDE per phi}
\Delta\varphi(x)\le c_1|\nabla\varphi(x)|+c_2\varphi(x)\quad\mbox{in a neighborhood of $x_0$.}
\end{equation}
By the strong maximum principle, this implies that $\varphi\equiv0$ in a neighborhood of $x_0$. Hence the set
$$
\Omega'=\{x\in\Omega\colon \varphi(x)=0\}\subset\Omega
$$
is open, and, obviously, closed in $\Omega$; consequently, as $\Omega$ is connected, $\Omega'$ coincides with $\Omega$. This implies that 
$$
\rank(D^2 v(x))\equiv r\quad\mbox{in $\Omega$.}
$$  

Adopting the notation of \cite{Korevaar-Lewis} (and previously introduced in \cite{Caffarelli-Friedman}), given two functions $f$ and $g$ in $C^1(\Omega)$, and $x\in\Omega$, we write:
$$
f(x)\lesssim g(x)
$$
if there exist $c_1$ $c_2>0$, such that
$$
f(x)-g(x)\le c_1|\nabla\varphi(x)|+c_2\varphi(x).
$$
The notation $f(x)\sim g(x)$ means that $f(x)\lesssim g(x)$ and $g(x)\lesssim f(x)$. If we omit the variable $x$ in one of the previous relations, we indicate that the relation holds for every $x$ in a neighborhood of $x_0$, with constants $c_1$ and $c_2$ independent of $x$. 

Hence \eqref{PDE per phi} can be written as
\begin{equation}\label{PDE per phi 2}
\Delta\varphi\lesssim0.
\end{equation}
In what follows, ``locally'' means in a suitable neighborhood of $x_0$. The matrix $D^2 v(x_0)$ has $r$ positive eigenvalues (and $(n-r)$ zero eigenvalues). Hence we may assume that in a neighborhood $\mathcal U$ of $x_0$, $r$ eigenvalues of $D^2v(x)$ are bounded from below by some a constant $c>0$. Let $z\in{\mathcal U}$, and choose a coordinate system such that $D^2 v(z)$ is a diagonal matrix (note that the equation verified by $v$ is invariant under rotations). Then we may assume that
$$
v_{jj}(z)\ge c,\quad j=1,\ldots,r.
$$
In what follows, we will use some formulas obtained in the proof of Theorem 1 in \cite{Korevaar-Lewis}. A comment is in order: the equation considered in \cite{Korevaar-Lewis} is different from \eqref{EDP per v}. On the other hand, the formulas that we will use, involving the derivatives of the solution $v$, are independent of the equation.

We set $G=\{1,\ldots,r\}$ and $B=\{r+1,\dots,n\}$. The following relation is (3.6) in \cite{Korevaar-Lewis}:
\begin{equation}\label{KL}
\Delta\varphi\sim Q\sum_{i\in B}(\Delta v)_{ii}-2\sum_{i\in B}\sum_{j\in G}|\nabla v_{ij}|^2 Q_j-R\sum_{i,j\in B}|\nabla u_{ij}|^2, 
\end{equation} 
where:
$$
Q=\prod_{j\in G} v_{jj},\quad Q_j=\frac{Q}{v_{jj}}, \; j\in G,\quad R=\sum_{j\in G} Q_j.
$$
Note that, by the convexity of $v$, $Q\ge0$, $Q_j\ge0$, for every $j\in G$, $R\ge 0$. Consequently
\begin{equation}\label{KL2}
\Delta\varphi\lesssim Q\sum_{i\in B}(\Delta v)_{ii}. 
\end{equation} 

As established in \cite{Korevaar-Lewis}, (3.8),
\begin{equation}\label{KL3}
\sum_{i\in B}(F( v,\nabla v))_{ii}\sim F_{tt}(v,\nabla v)\sum_{i\in B}v_i^2.
\end{equation}
Now let $H\colon\Omega\to\R$ be defined by
$$
H(x)=\langle\nabla v(x),x\rangle.
$$
Then, for every $i=1,\ldots,n$, 
$$
H_{ii}=2v_{ii}+\sum_{k=1}^nx_k v_{kii}.
$$
Therefore
\begin{equation}\label{KL4}
\sum_{i\in B}H_{ii}=2\sum_{i\in B} v_{ii}+\sum_{k=1}^n x_k\sum_{i\in B}v_{kii}.
\end{equation}
On the other hand (see again \cite{Korevaar-Lewis}, (3.3) and (3.4))
\begin{equation}\label{KL5}
\sum_{i\in B} v_{ii}\sim0,\quad \sum_{i\in B}v_{kii}\sim0,
\end{equation}
for every $k=1,\ldots,n$. 

Now \eqref{PDE per phi 2} follows from \eqref{KL2}, \eqref{KL3}, \eqref{KL4} and \eqref{KL5}.
\end{proof}

\begin{proof}[Proof of Theorem \ref{thm log-concavity lambda}]
Set $v=-\log u$ and write the equation for $v$:
$$
\Delta v= \lambda_\gamma+|\nabla v|^2+\langle x,\nabla v\rangle\,.
$$
Notice that 
$$
\Delta v=-\frac{\Delta u}{u}+\frac{|\nabla u|^2}{u}>0\quad\text{in }\Omega\,,
$$
thanks to assumption \eqref{Deltamin0}. Then the proof is a straightforward consequence of the combination of Theorem \ref{thm log-concavity lambda 3}, Proposition \ref{teo rango costante} and Proposition \ref{boundary behavior}.
\end{proof}

\section{The equality conditions of Theorem \ref{thm BM for lambda}}\label{section equality conditions}

Hereafter we will discuss the equality case of inequality \eqref{BM for lambda}, first giving a more general result, then obtaining as consequence an intermediate result, which in turn has Theorem \ref{equality conditions} as a corollary. 

To this aim, we restrict to a special class of convex sets, so that we can exploit Theorem \ref{thm log-concavity lambda} and follow similar ideas to those used in \cite[Section 5]{Colesanti} for the standard case, that is, when the Gauss measure is replaced by the Lebesgue measure.

\begin{theorem}\label{thm BM for lambda restricted} Let $\Omega_0, \Omega_1$ be an open, bounded and convex subset of $\R^n$, with boundary of class $C^{2,\alpha}_+$, for some $\alpha\in(0,1)$. Let $u_0$ and $u_1$ be solutions of \eqref{EDP lambda} with $\Omega=\Omega_0$ and $\Omega=\Omega_1$, respectively, and assume that 
\begin{equation}\label{Delta0}
\Delta u_0<0\,\,\text{ in }\Omega_0\,\,\text{ and }\,\,\Delta u_1<0\,\,\text{ in }\Omega_1\,.
\end{equation}

Let $t\in[0,1]$ and set
$$
\Omega_t=(1-t)\Omega_0+t\Omega_1.
$$
Then, equality holds in \eqref{BM for lambda}, i.e.
\begin{equation}\label{BM for lambda restricted}
\lambda_{\gamma}(\Omega_t)= (1-t)\lambda_{\gamma}(\Omega_0)+t\lambda_{\gamma}(\Omega_1),
\end{equation}
if and only if $\Omega_0$ and $\Omega_1$ coincide up to a translation.
\end{theorem}

\begin{proof} First of all, notice that, by \eqref{Ru_t}, if \eqref{BM for lambda restricted} holds, then the function $u_t$ defined in \eqref{ut} must be an eigenfunction of $\Omega_t$, that is, it must solve \eqref{EDP lambda} in $\Omega_t$.

For $i=0,1$, let $u_i$ be an eigenfunction of $\Omega_i$, that is a  solution of problem \eqref{EDP lambda} in $\Omega_i$. 
We set
$$
w_i=-\ln(u_i).
$$
Then $w_i$ verifies
\begin{equation}
\left\{
\begin{array}{lll}
\Delta w_i(x)=\lambda_i+|\nabla w_i(x)|^2+\langle \nabla w_i(x),x\rangle\quad\mbox{in $\Omega_i$,}\\
\\
\lim_{x\to\partial\Omega_i} w_i(x)=+\infty.
\end{array}
\right.
\end{equation}
Here, for simplicity, we denote by $\lambda_i$ the eigenvalue of $\Omega_i$, for $i=0,1,t$.By Theorem \ref{thm log-concavity lambda 3}, we know that $w_i$ is convex, then the boundary condition implies
\begin{equation}\label{gradient image}
\nabla w_i(\R^n)=\{\nabla w_i(x)\colon x\in\R^n\}=\R^n.
\end{equation}
By Theorem \ref{thm log-concavity lambda}, we further know that $w_i$ is strongly convex in $\Omega_i$, i.e. $D^2 w_i(x)>0$ for every $x\in\Omega_i$. Now we consider the standard Legendre conjugate of $w_i$:
$$
v_i(y)=w_i^*(y)=\sup_{x\in\Omega_i}\langle x,y\rangle-w_i(x)
$$
(note that the superscript $^*$  for functions has been already used with a different meaning in Section \ref{section viscosity}).

By \eqref{gradient image}, the domain of $v_i$ is $\R^n$. The strict positivity of $D^2 w_i$ on $\Omega_i$ implies that $v_i\in C^2(\R^n)$, moreover $\nabla w_i$ and $\nabla v_i$ are diffeomorpohisms between $\Omega_i$ and $\R^n$, inverse to each other, and finally:
\begin{equation}\label{hessian conjugate}
D^2 w_i(x)=\left(D^2 v_i(\nabla w_i(x))\right)^{-1}.
\end{equation}
We set
\begin{equation}\label{wt}
w_t=((1-t)v_0+tv_1)^*=((1-t)w_0^*+tw_1^*)^*.
\end{equation}
The function $w_t$ is defined and convex in $\Omega_t$ and, as it is well known, can be also defined as inf-convolution of $w_0$ and $w_1$, that is, 
\begin{equation}\label{sup convolution}
w_t(x_t)=\inf\{(1-t)w_0(x_0)+tw_1(x_1)\colon x_0\in\Omega_0,\, x_1\in\Omega_1,\, x_t=(1-t)x_0+t x_1\}
\end{equation}
for every $x_t\in\Omega_t:=(1-t)\Omega_0+t\Omega_1$.
Then we see that 
\begin{equation}\label{wtut}
w_t=-\log(u_t)\,,
\end{equation}
where $u_t$ is the function defined in \eqref{ut}.

By the boundary condition verified by $w_0$ and $w_1$, 
\begin{equation}\label{bc for wt}
\lim_{x\to\partial\Omega_t}w_t(x)=+\infty.
\end{equation}
Moreover, as $v_i\in C^2(\R^n)$ and $D^2 v_i>0$ in $\R^n$ for $i=0,1$, we have that $v_t:=(w_t)^*\in C^2(\R^n)$ and $D^2v_t>0$ in $\R^n$.

By \eqref{sup convolution} and \eqref{bc for wt}, for every $x_t\in\Omega_t$ there exist a (unique) $x_0\in\Omega_0$, $x_1\in\Omega_1$, such that
\begin{eqnarray*}
x_t=(1-t)x_0+tx_1,\quad w_t(x_t)=(1-t)w_0(x_0)+tw_1(x_1).
\end{eqnarray*}
Using the Lagrange multipliers theorem (see also \cite[pp. 126--127]{Colesanti}), it follows that
$$
\nabla w_t(x_t)=\nabla w_1(x_1)=\nabla w_0(x_0).
$$
Moreover, by \eqref{hessian conjugate} and \eqref{wt} (see also \cite[p. 127]{Colesanti}), we have
$$
D^2 w_t(x_t)=\left((1-t)(D^2 w_0(x_0))^{-1}+t(D^2 w_1(x_1))^{-1}\right)^{-1}.
$$

Next, we consider the function
$$
M\,\rightarrow\,\trace(M^{-1})
$$
as defined in the space of symmetric positive definite square matrices of order $n$. This function is convex (see, for instance, \cite[Appendix, Lemma 4]{Colesanti-Salani}, or \cite[Appendix]{Alvarez-Lasry-Lions}). Hence
\begin{equation}\label{Deltawt}
\Delta w_t(x_t)\geq(1-t)\Delta w_0(x_0)+t\Delta w_1(x_1)\,,
\end{equation}
that is
\begin{eqnarray*}
\Delta w_t(x_t)&\ge& (1-t)\lambda_0+t\lambda_1+(1-t)|\nabla w_0(x_0)|^2+t|\nabla w_1(x_1)|^2+\\
&&+(1-t)\langle\nabla w_0(x_0),x_0\rangle+t\langle\nabla w_1(x_1),x_1\rangle\,,
\end{eqnarray*}
which, by the previous considerations, implies
\begin{equation}\label{Deltawt2}
\Delta w_t(x_t)\ge(1-t)\lambda_0+t\lambda_1+|\nabla w_t(x_t)|^2+\langle\nabla w_t(x_t),x_t\rangle\,.
\end{equation}
On the other hand, as observed at the very beginning of the proof, the function 
$$
u_t=e^{-w_t},
$$
verifies
\begin{equation}\label{BVP for barut}
\Delta\bar u_t=-((1-t)\lambda_0+t\lambda_1)\bar u_t+\langle \nabla\bar u_t,x\rangle,\quad\mbox{in $\Omega_t$,}\\
\end{equation}
which implies the equality in \eqref{Deltawt2} and in turn in \eqref{Deltawt}. Then equality holds in the differential inequality in problem \eqref{BVP for barut}. Going backwards through the previous steps of the proof, we find out that, by equality conditions in \cite[Appendix, Lemma 4]{Colesanti-Salani},
$$
D^2 w_0(x_0)=D^2 w_1(x_1).
$$
As a further consequence we deduce
$$
D^2 v_0(y)=D^2 v_1(y)\quad\forall\, z\in\R^n
$$
and then
$$
\nabla v_0(y)=\nabla v_1(y)+x_0
$$
for some $x_0\in\R^n$. On the other hand:
$$
\Omega_i=\nabla v_i(\R^n), \quad i=0,1
$$
(as $\nabla v_i$ and $\nabla w_i$ are inverse to each other). We conclude that
$$
\Omega_0=\Omega_1+x_0.
$$  
\end{proof}

\begin{corollary}\label{cor1}
In Theorem \ref{thm BM for lambda restricted}  substitute assumption \eqref{Delta0} with the following:
\begin{equation}\label{maxu0u1}
\text{$u_0$ and $u_1$ have maximum at the origin}.
\end{equation}
Then the same conclusion holds.
\end{corollary}
\begin{proof}
As in Corollary \ref{max0}, we have just to observe that assumption \eqref{maxu0u1}, together with the convexity of the involved sets, yield assumption \eqref{Delta0}.
\end{proof}
\begin{proof}[Proof of Theorem \ref{equality conditions}]
Like in Corollary \ref{symmetric}, observe that the symmetry of $\Omega_0$ and $\Omega_1$ implies \eqref{maxu0u1}, then we can apply Corollary \ref{cor1} to obtain that there exist $x_0$ such that $\Omega_1=\Omega_0+x_0$.
As $\Omega_0$ and $\Omega_1$ are origin symmetric, then $x_0=0$.
\end{proof}

\end{document}